\documentclass{amsart}

\usepackage{amssymb, amsfonts}
\usepackage{amsmath}
\usepackage{xcolor}
\usepackage{mathtools}
\usepackage{enumerate}
\usepackage{stmaryrd}

\newtheorem{thm}{Theorem}[section]
\newtheorem{prop}[thm]{Proposition}
\newtheorem{lem}[thm]{Lemma}
\newtheorem{cor}[thm]{Corollary}

\theoremstyle{definition}
\newtheorem{definition}[thm]{Definition}

\theoremstyle{remark}

\newtheorem{remark}[thm]{Remark}

\numberwithin{equation}{section}

\newcommand{\R}{\mathbb{R}}

\begin{document}


\title[Error estimates of RKDG schemes for fractional conservation laws]{A priori error estimates of Runge-Kutta discontinuous Galerkin schemes to smooth solutions of fractional conservation laws}


\author{Fabio Leotta}
\address{Department of Mathematics, Technical University Darmstadt, 
Germany}
\email{leotta@mathematik.tu-darmstadt.de}


\author{Jan Giesselmann}
\address{Department of Mathematics, Technical University Darmstadt, 
	Germany}
\email{giesselmann@mathematik.tu-darmstadt.de}



\begin{abstract}
We give a priori error estimates of second order in time fully explicit Runge-Kutta discontinuous Galerkin schemes using upwind fluxes to smooth solutions of scalar fractional conservation laws in one space dimension. Under the time step restrictions $\tau\leq c h$ for piecewise linear and $\tau\lesssim h^{4/3}$ for higher order finite elements, we prove a convergence rate for the energy norm $\|\cdot\|_{L^\infty_tL^2_x}+|\cdot|_{L^2_tH^{\lambda/2}_x}$ that is optimal for solutions and flux functions that are smooth enough. Our proof relies on a novel upwind projection of the exact solution.
\end{abstract}


 \maketitle

\tableofcontents

\section{Introduction}

We consider the fractional conservation law
\begin{equation}\label{FCL}
\begin{cases}
\partial_t u (t,x) + \partial_x f(u(t,x)) = g_\lambda[u](t,x)\qquad &(t,x)\in Q_T:=[0,T]\times\R\\
u(0,x)=u_0(x)\qquad &x\in\R,
\end{cases}
\end{equation}
where $g_\lambda$ is the nonlocal fractional Laplace operator $-(-\partial^2_x)^{\lambda/2}$ for some $\lambda\in(0,1)$. The operator can be defined by Fourier transform via
\begin{equation}
\widehat{g_\lambda[\varphi]}(y)=-|y|^\lambda\widehat{\varphi}(y),
\end{equation}
or, equivalently, by the integral formula of Droniou and Imbert \cite{DroniouImbert06},
\begin{equation}
g_\lambda[\varphi](x)=c_\lambda\int_\R\frac{\varphi(x+z)-\varphi(x)}{|z|^{1+\lambda}}~dz,
\end{equation}
where $c_\lambda>0$ is a constant that depends solely on $\lambda$.

While the standard Laplacian $\Delta$ can be used to describe diffusion of particles that are governed by a Gaussian probability density $p(dx)$ for a displacement of $dx$ in an incremental time step, the fractional Laplacian $g_\lambda$ can be used to model diffusion processes that are governed by L\`evy probability densities; cf.~\cite{AbeThurner05} and \cite{SimonOlivera17}. Physically motivated applications may range from molecular biology and radiation \cite{Bouchaud95} to hydrodynamics \cite{WeeksSolomonUrbachSwinney95}; see also \cite{TankCont15} for applications in mathematical finance. 

In the case $\lambda\in(1,2)$, it has been shown in \cite{DroniouGallouetVovelle02} that (\ref{FCL}) turns $L^\infty$ initial data into a unique smooth solution as for the classical case $\lambda=2$. This general regularizing effect, however, is lost for $\lambda\in(0,1)$ and the solution may develop shocks in finite time (cf.~\cite{AlibaudDroniouVovelle07}) even for smooth initial data. Accordingly, an entropy formulation for fractional conservation laws has been developed by Alibaud \cite{Alibaud06} that guarantees well-posedness in the $L^\infty$ framework for $\lambda\in(0,2)$.

Regarding numerical methods for (\ref{FCL}), Droniou \cite{Droniou10} was the first\footnote{There have been some probabilistic approximation approaches for $\lambda>1$ before, as in \cite{JourdainMeleardWoyczynski05}, where the conservation law is differentiated to derive a stochastic interpretation.} who considered a (fully discrete) finite volume method for $\lambda\in(0,2)$ and proved convergence (for $L^\infty\cap \operatorname{BV}$ initial data) to the unique entropy solution but did not recover any convergence rate. In \cite{CifaniJakobsen14} Cifani and Jakobsen have been able to give convergence rates in the $L^1$-norm for a low order numerical method in the general case $\lambda\in(0,2)$. In \cite{CifaniJakobsenKarlsen10}, Cifani et al.~propose a discontinuous Galerkin (DG) method for $\lambda\in(0,1)$ and prove convergence in the $L^\infty L^1$-norm for an implicit-explicit and a fully explicit finite volume scheme (for $L^1\cap\operatorname{BV}$ initial data) with order $\sqrt{h}$ and $\min\{\sqrt{h},h^{1-\lambda}\}$, respectively. Additionally, they proved that in the case of linear fluxes the semidiscrete DG method converges in $L^\infty L^2$ with order $h^{k+\frac{1}{2}}$ provided the exact solution is smooth. Similarly, for $\lambda\in(1,2)$, Xu and Hesthaven propose a semi-discrete discontinuous Galerkin method in \cite{XuHesthaven14} and prove convergence in the $L^2$-norm with order $h^{k+\frac{1}{2}}$.

Since the convective part of fractional conservation laws (\ref{FCL}) with $\lambda\in(0,1)$ seems to dominate the dynamics, it is natural to compare to established results from the purely hyperbolic community.  As has been proven by Zhang and Shu in \cite{ZhangShu06}, explicit second order in time Runge-Kutta discontinuous Galerkin (RKDG) methods using upwind fluxes for purely hyperbolic conservation laws $\partial_t u + \partial_x f(u)=0$, converge with the optimal order of $h^{k+1}+\tau^2$ as long as the solution is smooth enough. In this paper, we generalize this high order accuracy to RKDG methods for the fractional conservation law (\ref{FCL}). 

Although the appearance and persistence of shocks in the exact solution of a fractional conservation law may inhibit optimal convergence rates for higher order methods, one expects better convergence properties away from the shocks where the solution is regular. In this sense, it is of interest to investigate not only worst case but also best case approximation results for such non-regularizing PDEs.

A short overview of this paper is given as follows. In section 2 we state our main result, the convergence theorem, for the second order total variation diminishing RKDG scheme that is stated in section 3. In section 4 we present the proof of our main result. Here we establish the central energy identity for the numerical scheme, introduce a novel upwind projection operator and show how it facilitates optimal estimates for the right-hand side of the energy identity; it culminates in the application of a modified Gronwall argument which yields the convergence rates stated in our main result.  
 
 \section{Main result}
In this paper, we propose a second order in time fully explicit RKDG method for (\ref{FCL}) and prove the following a priori estimate.  

\begin{thm}\label{thm:main}
	Let $\alpha\in\mathbb{N}$ with $\alpha\geq 3$, $f\in C^{\alpha+1}(\mathbb{R})$ and $u\in C^\alpha([0,T];H^{k+1}(\mathbb{R}))$ be the exact solution to the fractional conservation law (\ref{FCL}). Let $u_h^n$, a piecewise discontinuous polynomial of degree $k\geq 1$, be the numerical solution of the RKDG scheme from section \ref{sec:RKDG}, at time level $t^n$, with CFL condition $\tau\leq c h$ for a certain $c>0$ in the case $k=1$ and $\tau\lesssim h^{4/3}$ for $k\geq2$. Then we have
	\begin{equation}\label{eq:mainres}
	\|u^n-u^n_h\|_{L^2(\R)} + \left(\sum_{m=0}^{n-1}\tau|u^m-u^m_h|^2_{H^{\lambda/2}(\mathbb{R})}\right)^{1/2}\lesssim h^{k+1-\frac{1}{\alpha}}+h^{k+1-\frac{\lambda}{2}}+\tau^2,
	\end{equation}
	for all $n\in\{1,\ldots,N\}$, where $u^n:=u(t^n,\cdot)$.
\end{thm}
Note that we write $x\lesssim y$ if there exists a constant $a>0$ independent of $n$,$h$ and $\tau$ such that $x\lesssim ay$.

\begin{remark}
	The convergence rate is optimal for $\alpha\geq 2/\lambda$ due to the $H^{\lambda/2}$-norm on the left-hand side of (\ref{eq:mainres}).
\end{remark}

\begin{remark}
	The regularity assumption $u\in C^\alpha([0,T];H^{k+1}(\mathbb{R}))$ is chosen for convenience as it implies all the conditions used in the proofs of this paper, namely $u\in C^3([0,T];H^1(\R))$ for consistency (cf.~Lemma \ref{lem:consistency}), $u\in H^1(0,T;H^{k+1}(\R))$ to control (time step differences of) projection errors (cf.~(\ref{eq:tsdifference1})-(\ref{eq:tsdifference2})) and lastly $u\in C^\alpha([0,T];C_b(\R))$ to quantitatively control sign changes in $f'(u)$ (cf.~Lemma \ref{lem:O_j}). In particular, the latter regularity assumption is the only one added to those made in the setting of purely hyperbolic conservation laws in  \cite{ZhangShu06}.
\end{remark}

\section{The RKDG method}\label{sec:RKDG}
Let us introduce (for simplicity) the equidistant space grid $x_j=jh,~j\in\mathbb{Z}$ and let us label the cells as $I_j=(x_j,x_{j+1})$. The space of polynomials of degree at most $k\in\{1,2,3,\ldots\}$ with defined on the interval $I_j$ is denoted by $P^k(I_j)$ and we will work with
\begin{equation}
V_h=V^k_h=\{\phi\in L^2(\R)~:~\phi|_{I_j}\in P^k(I_j)~\forall j\in\mathbb{Z}\}
\end{equation}
as our finite element space. Further, the time grid is given by $t_n=n\tau$, where $n\in\{0,\ldots,N\}$ and $N\tau=T$.

Note that $V_h\subset H^{\lambda/2}(\R)$ with
\begin{equation}\label{eq:fracinv}
|\phi|^2_{H^{\lambda/2}(\R)}\lesssim h^{-\lambda}\|\phi\|^2_{L^2(\R)},
\end{equation}
for all $\phi\in V_h$; cf.~the appendix in \cite{CifaniJakobsenKarlsen10}.

Given $u^n_h\in V_h$, the approximate solution at time $t^{n+1}$ is given by the second order in time fully explicit RKDG scheme as follows: find $w^n_h,u^{n+1}_h\in V_h$ such that for all $j\in\mathbb{Z}$ and $p_h, q_h\in P^k(I_j)$ there holds
\begin{subequations}\label{scheme}
\begin{align}
	\int_{I_j}w^n_h p_h~dx =& \int_{I_j}u^n_h p_h~dx + \tau \mathcal{H}_j(u^n_h,p_h) + \tau\mathcal{D}_j(u^n_h,p_h)\\
	\int_{I_j}u^{n+1}_h q_h~dx =&~ \frac{1}{2}\int_{I_j}(u^n_h + w^n_h) q_h~dx + \frac{\tau}{2}\mathcal{H}_j(w^n_h,q_h) + \frac{\tau}{2}\mathcal{D}_j(w^n_h,q_h)
\end{align}
\end{subequations}
with the initial datum $u_h^0$ given by a suitable projection $\Pi^0_hu^0$; cf.~section \ref{GaussRadau}. The operators $\mathcal{H}_j:H^1(I_j)\times V_h\rightarrow\R$ and $ \mathcal{D}_j:H^{\lambda/2}(\R)\times H^{\lambda/2}(\R)\rightarrow\R$ are defined by
\begin{align}
\mathcal{H}_j(p,q) &:= \int_{I_j}f(p)q_x~dx - \hat{h}(p)_{j+1}q^-_{j+1} + \hat{h}(p)_{j}q^+_{j},\label{eq:H}\\
\mathcal{D}_j(p,q) &:= (g_\lambda[p],\mathbf{1}_{I_j}q)_{H^{-\lambda/2}(\R),H^{\lambda/2}(\R)},
\end{align}
where $\hat{h}(p)_j=\hat{h}(p^-_j,p^+_j)$ is an upwind numerical flux that depends on the two traces of the function $p$ at the boundary point $x_j$, namely $p^\pm_j=p(x^\pm_j)$. Note that by Fourier characterization of the fractional Sobolev norm it can be readily seen that the fractional Laplacian ${g_\lambda:H^{\lambda/2}(\R)\rightarrow H^{-\lambda/2}(\R)}$ is indeed bounded; cf.~\cite{NezzaPalatucciValdinoci12}.

We consider fluxes $\hat{h}(a,b)$ satisfying the following conditions:
\begin{enumerate}[(a)]
	\item It is locally Lipschitz continuous.
	\item It is consistent with the flux $f(p)$, i.e.~$\hat{h}(p,p)=f(p)$.
	\item It is monotone, i.e.~a nondecreasing function of its first argument and a nonincreasing function of its second argument.
\end{enumerate}
Recall that a numerical flux $\hat{h}(p)$ is called upwind if it satisfies
\begin{equation}
\hat{h}(p)=\begin{cases}
f(p^-)\quad &\text{if}~f'(q)\geq 0~\forall q\in[\min(p^-,p^+),\max(p^-,p^+)]\\ 
f(p^+)\quad &\text{if}~f'(q)< 0~\forall q\in[\min(p^-,p^+),\max(p^-,p^+)].
\end{cases}
\end{equation}

\subsection{The difference between numerical and physical flux}
In the upcoming proofs, we will need to measure the difference between the numerical flux and the physical flux. Following Zhang and Shu (cf.~Lemma 3.1 and (5.23) in \cite{ZhangShu06}), for any piecewise smooth function $p\in L^2(\R)$ we define on any boundary $x_j\in\R$ the quantity
\begin{equation}
a(p)_j\equiv a(p^+,p^-)_j:=\begin{cases}
\llbracket p\rrbracket_j^{-1}(f(\bar{p}_j)-\hat{h}(p)_j) &\text{if}~\llbracket p\rrbracket_j\neq 0\\
|f'(\bar{p}_j)| &\text{if}~\llbracket p\rrbracket_j=0,
\end{cases}
\end{equation}
where $\llbracket p\rrbracket_j=p^+_j - p^-_j$ and $\bar{p}_j=(p^+_j + p^-_j)/2$. We will drop the subscript if no confusion can arise.
\begin{lem}\label{lem:a}
	Let $f\in C^3(\R)$. Then $a(p)$ is non-negative and uniformly bounded for all $(p^+,p^-)\in\R^2$. Moreover, there holds
	\begin{subequations}
		\begin{align}
	\frac{1}{2}|f'(\bar{p})| &\leq a(p) + c_\ast|\llbracket p\rrbracket|\\
	-\frac{1}{8}f''(\bar{p})\llbracket p\rrbracket &\leq a(p) + c_\ast\llbracket p\rrbracket^2,
	\end{align}
	\end{subequations}
where the constant $c_\ast\geq0$ depends solely on the maximum of $|f''|$ and $|f'''|$.
\end{lem}
Note however, that the numerical flux $\hat{h}$ need not be upwind for Lemma \ref{lem:a} to hold; monotonicity of $\hat{h}$ and Taylor expansion of $f$ suffice.

Furthermore, we assume that there exist constants $c$ and $c_\ast$ independent of $f$ and $p$ such that
\begin{equation}\label{eq:flux}
	a(p)\leq c|f'(\overline{p})| + c_\ast|\llbracket p\rrbracket|,
\end{equation}
which, for example, is satisfied by the Godunov or Roe numerical flux, as can be seen by Taylor expansion.

\section{Analysis of the scheme}
In this section, we present the proof of our main result, Theorem \ref{thm:main}. We follow the ideas of Zhang and Shu \cite{ZhangShu06} in splitting the numerical error into a projection and a discrete error to then track the time evolution of the latter. Firstly, we thus establish consistency of the RKDG method and derive an energy identity for the evolving discrete error. After bounding the energy terms appropriately, the error estimate is obtained by a Gronwall argument. Of course, the choice of a suitable projection operator is necessary to derive an optimal convergence order. On each time level and in each element, the projection operator is chosen as the interpolation on Gauss-Radau quadrature points including either the right or the left boundary point. Depending upon the behavior of the exact solution on that time level, a judicious choice of the included boundary point is made that may reuse the choice made at the previous time level. This constitutes the main difference between our framework and \cite{ZhangShu06}, where the projection operator can be chosen to only depend on the initial datum $u_0$; cf.~Remark \ref{remark:GaussRadau}.

\subsection{Consistency, error equations and energy identity}\label{subsec:consis}In the following, we will abbreviate $\|\cdot\|=\|\cdot\|_{L^2(\R)}$ and $\Pi^n_h$ shall denote a generic time-dependent projection operator that will be specified later.

 Define
\begin{align*}
\xi^n_h &:= u^n_h - \Pi^n_hu^n, &\zeta^n_h &:= w^n_h - \Pi^n_hw^n,\\
\xi^n_\pi &:= u^n - \Pi^n_hu^n, &\zeta^n_\pi &:= w^n - \Pi^n_hw^n,
\end{align*}
with $w:=u+\tau\partial_tu$ and $w^n:=w(t^n,\cdot)$. Using these quantities, the numerical errors can be split as 
\begin{equation}\label{split}
e^n_u:=u^n-u^n_h = \xi^n_\pi - \xi^n_h,\qquad e^n_w:=w^n-w^n_h = \zeta^n_\pi - \zeta^n_h.
\end{equation}
In order to track the time evolution of the discrete errors $\xi^n_h,~\zeta^n_h$ we first need to establish 
\begin{lem}[Consistency]\label{lem:consistency}
	Let $f\in C^1(\R)$ and $u\in C^3([0,T];H^1(\R))$ be the exact solution to the fractional conservation law (\ref{FCL}).
	Then, for all $j\in\mathbb{Z}$ and $p_h, q_h\in P^k(I_j)$ there holds
	\begin{subequations}\label{eq:consistency}
		\begin{align}
		\int_{I_j}w^n p_h~dx =& \int_{I_j}u^n p_h~dx + \tau \mathcal{H}_j(u^n,p_h) + \tau\mathcal{D}_j(u^n,p_h)\\
		\int_{I_j}u^{n+1} q_h~dx =&~ \frac{1}{2}\int_{I_j}(u^n+w^n) q_h~dx + \frac{\tau}{2}\mathcal{H}_j(w^n,q_h) +\frac{\tau}{2}\mathcal{D}_j(w^n,q_h)\\ &~ +  \int_{I_j}E^n q_h~dx,\nonumber
		\end{align}
	\end{subequations}
	with $\|E^n\|\lesssim \tau^3$.
	\begin{proof}
		By Taylor expansion in time we have
		\begin{equation*}
		u(t+\tau,x) - u(t,x) - \frac{1}{2}\tau\partial_t u(t,x) - \frac{1}{2}\tau\partial_t u(t+\tau,x) = O(\tau^3). 
		\end{equation*}
		Note that indeed $u(\cdot,x)\in C^3([0,T])$ for all $x\in\R$ due to Morrey's inequality.
		On the other hand, the fractional conservation law (\ref{FCL}) and Taylor expansion yield
		\begin{align*}
		\partial_t u (t+\tau,x) &= - (\partial_x w(t,x) + \partial_x R(t,x))f'(w(t,x)+R(t,x)) + g_\lambda[w](t,x) + g_\lambda[R](t,x)\\
		&= -\partial_x f(w(t,x)) + g_\lambda[w](t,x) + O(\tau^2),
		\end{align*}
		where $R$ is the Taylor residual of the first order expansion.
		In total, we thus get
		\begin{equation*}
		u(t+\tau,x) = \frac{1}{2}u(t,x) + \frac{1}{2}w(t,x) - \frac{1}{2}\tau\partial_x f(w(t,x)) + \frac{1}{2}\tau g_\lambda[w(t,x)] + O(\tau^3),
		\end{equation*}
		and therefore
		\begin{align*}
		w^n &= u^n - \tau\partial_x f(u^n) + \tau g_\lambda[u^n]\\
		u^{n+1} &= \frac{1}{2}u^n + \frac{1}{2}w^n - \frac{1}{2}\tau\partial_x f(w^n) + \frac{1}{2}\tau g_\lambda[w^n] + E^n.
		\end{align*}
		Note that the consistency error $E^n$ is integrable and of order $O(\tau^3)$ in the $L^2(\R)$-norm due to the regularity assumptions on $f$ and $u$.
		What is left to do now, is to multiply the above two equations by test functions and integrate over $I_j$. After an integration by parts, we can conclude the proof by consistency of the numerical flux.
	\end{proof}
\end{lem}

\begin{remark}
	In the proof of Lemma \ref{lem:consistency} we have made use of the fact that, by the maximum principle (cf.~Alibaud \cite{Alibaud06}), $f$ can be replaced by a function with compact support without changing the dynamics of the fractional conservation law (\ref{FCL}). This will also be tacitly applied in subsequent proofs.
\end{remark}

Since the numerical error can be split as (\ref{split}), the scheme (\ref{scheme}) together with its consistency (\ref{eq:consistency}) yield the 
\begin{lem}[Error equations]
	For all $j\in\mathbb{Z}$ and $p_h,q_h\in P^k(I_j)$ there holds
	\begin{subequations}\label{eq:erroreq}
		\begin{align}
		\int_{I_j}\zeta^n_h p_h~dx =& \int_{I_j}\xi^n_h p_h~dx + \mathcal{K}^n_j(p_h)+ \tau\mathcal{D}_j(u^n_h-u^n,p_h)\label{eq:erroreqa}\\
		\int_{I_j}\xi^{n+1}_h q_h~dx =& ~\frac{1}{2}\int_{I_j}(\xi^n_h+\zeta^n_h)q_h~dx + \frac{1}{2}\mathcal{L}^n_j(q_h) + \frac{\tau}{2}\mathcal{D}_j(w^n_h-w^n,q_h),\label{eq:erroreqb}
		\end{align}
	\end{subequations}
	where
	\begin{align}
	\mathcal{K}^n_j(p_h) &:= \int_{I_j}(\zeta^n_\pi - \xi^n_\pi)p_h~dx + \tau\mathcal{H}_j(u^n_h,p_h) - \tau\mathcal{H}_j(u^n,p_h)\label{eq:K}\\
	\mathcal{L}^n_j(p_h) &:= \int_{I_j}(2\xi^{n+1}_\pi-\zeta^n_\pi-\xi^n_\pi-2E^n)p_h~dx + \tau\mathcal{H}_j(w^n_h,p_h) - \tau\mathcal{H}_j(w^n,p_h).\label{eq:L}
	\end{align}
\end{lem}

For the next proof we will use the fact that there exists $c_\lambda>0$ such that
	\begin{equation}\label{eq:sym}
	\int_\R g_\lambda[\phi]\phi~dx=-\frac{c_\lambda}{2}|\phi|^2_{H^{\lambda/2}(\R)},
	\end{equation}
for all $\phi\in H^{\lambda/2}(\R)$; cf.~the appendix in \cite{CifaniJakobsenKarlsen10}.

From the error equations (\ref{eq:erroreq}) we derive the central

\begin{lem}[Energy identity]
	Denoting $\mathcal{K}^n:=\sum_{j\in\mathbb{Z}}\mathcal{K}^n_j$, $\mathcal{L}^n:=\sum_{j\in\mathbb{Z}}\mathcal{L}^n_j$ and $D:=\sum_{j\in\mathbb{Z}}D_j$, there holds
	\begin{align}\label{eq:energyidentity}
	\|\xi^{n+1}_h\|^2 - \|\xi^n_h\|^2 =& \|\xi^{n+1}_h - \zeta^n_h\|^2 + \mathcal{K}^n(\xi^n_h) + \mathcal{L}^n(\zeta^n_h) - \tau (\mathcal{D}(\zeta_\pi^n,\zeta^n_h) + \mathcal{D}(\xi_\pi^n,\xi^n_h))\nonumber\\&~- \tau \frac{c_\lambda}{2}(|\xi_h^n|^2_{H^{\lambda/2}(\R)} + |\zeta_h^n|^2_{H^{\lambda/2}(\R)}).
	\end{align}
\end{lem}
\begin{proof}
	Testing (\ref{eq:erroreqb}) with $\zeta^n_h$ and (\ref{eq:erroreqa}) with $\xi^n_h$ before summing over $j\in\mathbb{Z}$, one gets
	\begin{align*}
	2\int_{\R}\xi^{n+1}_h\zeta^n_h~dx =& \int_{\R}\xi_h^n\zeta_h^n~dx + \|\zeta_h^n\|^2 + \mathcal{L}^n(\zeta_h^n) + \tau\mathcal{D}(w^n_h-w^n,\zeta^n_h)\\
	=&~ \|\xi^n_h\|^2 + \|\zeta_h^n\|^2 + \mathcal{K}^n(\xi^n_h) + \mathcal{L}^n(\zeta^n_h)\\ &+ \tau(\mathcal{D}(u^n_h-u^n,\xi^n_h) + \mathcal{D}(w^n_h-w^n,\zeta^n_h)).
	\end{align*}
	On the other hand, the binomial formula
	\begin{equation*}
	\|\xi^{n+1}_h-\zeta^n_h\|^2=\|\xi^{n+1}_h\|^2 - 2\int_{\R}\xi^{n+1}_h\zeta^n_h~dx + \|\zeta^n_h\|^2,
	\end{equation*}
	holds and we conclude the proof by (\ref{eq:sym}) after rearranging terms.
\end{proof}

\begin{remark}\label{remark:GaussRadau}
By a Gronwall argument, bounding the right hand side of the energy identity (\ref{eq:energyidentity}) will yield a bound on the discrete error $\xi^n_h$. However, to produce an optimal error estimate, we must choose a \textit{suitable} projection operator $\Pi^n_h$. For example, taking the standard $L^2$-projection in each time step will not suffice to deal with the fluxes between elements and one is forced to sacrifice a power of $h^{1/2}$ by use of an inverse inequality for the operator $\mathcal{H}$ in (\ref{eq:H}). On the other hand, changing the projection each time step to control the fluxes between elements will increase the error bound due to the time step difference $\xi^{n+1}_\pi - \xi^n_\pi$ occuring in (\ref{eq:L}); cf.~(\ref{eq:tsdifference2}) and the paragraph thereafter. In the setting of purely hyperbolic conservation laws, Zhang and Shu chose to work with  what they call the Gauss-Radau projection $\R_h$ in \cite{ZhangShu06} to handle the fluxes between elements. However, since the solution is conserved along characteristics, the projection operator indeed does not change in time. Let us briefly discuss why the strategy does not work in the presence of a diffusion term as in the fractional conservation law (\ref{FCL}):

In \cite{ZhangShu06}, the projection at time level $t^n$, by definition depends on the exact solution $u^n=u(t^n,\cdot)$ of the conservation law $\partial_t u +\partial_x f(u)=0$ and is defined locally, namely
\begin{itemize}
	\item if $f'(u^n) \geq 0$ on $I_j$, interpolation on Gauss-Radau points including the right boundary point is used,
	\item if $f'(u^n) \leq 0$ on $I_j$, interpolation on Gauss-Radau points including the left boundary point is used,
	\item if $f'(u^n)$ has a zero in $I_j$, the standard $L^2$-projection is used.
\end{itemize}
If now initially $f'(u^0(x^\ast))=0$ for some $x^\ast\in I_j$, one has $f'(u^n(x^\ast))=0$ for all times $t^n$, due to the characteristic being a vertical line and $u$ being constant along characteristics. Furthermore, since characteristics cannot cross due to smoothness of $f$ and $u$, this means that the Gauss-Radau projection is actually not time dependent at all.

Now, it is clear that the characteristics argument breaks down in the presence of a diffusion term and thus in particular, the Gauss-Radau projection $\R_h$ as defined above for $u$ being the exact solution to the fractional conservation law (\ref{FCL}), must be expected to change each time step.
\end{remark}

\subsection{The upwind projection operator}\label{GaussRadau}
In this section we will introduce the upwind projection operator that is related to the projection employed by Zhang and Shu \cite{ZhangShu06} in the setting of purely hyperbolic conservation laws. In the context of fractional conservation laws, taking the same projection operator as in \cite{ZhangShu06} is not viable since time differences of projection errors would introduce sub-optimal error estimates; cf.~Remark \ref{remark:GaussRadau}.

Let us first define the left and right Gauss-Radau projections, such that at each time level, the upwind projection will be given by a judicious choice of the Gauss-Radau projections.

\begin{definition}
	The left and right Gauss-Radau projection $\Pi^{\text{left},\hspace{0.05cm}j}_{h}v$ and $\Pi^{\text{right},\hspace{0.05cm}j}_{h}v\in P^k(I_j)$, respectively, of $v\in C(\overline{I_j})$ are the unique polynomials of degree $k$ such that
	\begin{equation*}
	 \Pi^{\text{left},\hspace{0.05cm}j}_{h}v(x_{j}^+)=v(x_{j}^+),\quad \Pi^{\text{right},\hspace{0.05cm}j}_{h}v(x_{j+1}^-)=v(x_{j+1}^-)
	\end{equation*}
	and
	\begin{equation}\label{eq:orth}
	\int_{I_j}(v-\Pi^{\text{left},\hspace{0.05cm}j}_{h}v)p_h =\int_{I_j}(v-\Pi^{\text{right},\hspace{0.05cm}j}_{h}v)p_h=0\quad\forall p_h\in P^{k-1}(I_j).
	\end{equation}
\end{definition}

For functions at time level $t^{n}$, the upwind projection operator $\Pi^n_h$ depends on the exact solution $u^{n}$ and possibly $\Pi^{n-1}_h$. The definition is given inductively as follows.

\begin{definition}[Upwind projection operator]\label{def:radau}
	Initially, we define the projection $\Pi^{0,\hspace{0.05cm}j}_{h} v\in P^k(I_j)$ of a function $v\in C(\overline{I_j})$ on any given element $I_j$ as follows:
	\begin{itemize}
		\item If $f'(u^0)> 0$ on $I_j$, set $\Pi^{0,\hspace{0.05cm}j}_{h}:=\Pi^{\text{right},\hspace{0.05cm}j}_{h}$.
		\item Else, set $\Pi^{0,\hspace{0.05cm}j}_{h}:=\Pi^{\text{left},\hspace{0.05cm}j}_{h}$.
	\end{itemize}
	
	At time level $t^n$, the projection $\Pi^{n,\hspace{0.05cm}j}_{h} v\in P^k(I_j)$ of $v\in C(\overline{I_j})$ is defined as follows:
	\begin{itemize}
		\item If $f'(u^n) > h$ on $I_j$, set $\Pi^{n,\hspace{0.05cm}j}_{h}:=\Pi^{\text{right},\hspace{0.05cm}j}_{h}$.
		\item If $f'(u^n) < - h$ on $I_j$, set $\Pi^{n,\hspace{0.05cm}j}_{h}:=\Pi^{\text{left},\hspace{0.05cm}j}_{h}$.
		\item Else, the projection of the previous time step is reused, i.e.~we set ${\Pi^{n,\hspace{0.05cm}j}_{h}:=\Pi^{n-1,\hspace{0.05cm}j}_{h}}$.
	\end{itemize}

The global projection $\Pi^n_h(v)\in V_h$ for $v\in C(\R)$ is then defined accordingly by setting $\Pi^n_h(v)|_{I_j}:=\Pi^{n,\hspace{0.05cm}j}_{h}(v)$.
\end{definition}

By accommodating a tolerance w.r.t.~$h$, this projection operator allows us to derive an upper bound to the number of times $t^n$ for which ${\Pi^{n,\hspace{0.05cm}j}_{h}\neq\Pi^{n-1,\hspace{0.05cm}j}_{h}}$.

Note that by a standard Bramble-Hilbert argument the projection operator yields the quasi-optimal approximation property,
\begin{equation}\label{eq:bestappr}
\|\xi^n_\pi\| + h\|\xi^n_\pi\|_{L^\infty(\R)} + h^{1/2}|\xi^n_\pi|_{\Gamma_h} + h^{\lambda/2}|\xi^n_\pi|_{H^{\lambda/2}(\R)}\lesssim h^{k+1},
\end{equation}
for all $u\in H^{k+1}(\R)$, where $|u|_{\Gamma_h}:=\left(\sum_{j\in\mathbb{Z}}\llbracket u\rrbracket^2_j\right)^{1/2}$.

\subsection{Time step differences of projection errors}\label{subsection:time step difference}
By linearity and Jensen's inequality it is clear that
\begin{equation}\label{eq:tsdifference1}
\|\xi_\pi^{n+1} - \xi^n_\pi\|^2_{L^2(I_j)} \lesssim h^{2k+2}\tau\|\partial_t u\|^2_{L^2(t^n,t^{n+1};H^{k+1}(I_j))},
\end{equation}
for all subsequent times at which $\Pi^{n+1,\hspace{0.05cm}j}_{h}=\Pi^{n,\hspace{0.05cm}j}_{h}$. However, if $\Pi^{n+1,\hspace{0.05cm}j}_{h}\neq\Pi^{n,\hspace{0.05cm}j}_{h}$, one is left with only
\begin{equation}\label{eq:tsdifference2}
\|\xi_\pi^{n+1} - \xi^n_\pi\|^2_{L^2(I_j)} \lesssim h^{2k+2}(\| u^0\|^2_{H^{k+1}(I_j)}+T\|\partial_t u\|^2_{L^2(0,T;H^{k+1}(I_j))}),
\end{equation}
in general. 

\begin{remark}
	The prefactor $\| u^0\|^2_{H^{k+1}(I_j)}+T\|\partial_t u\|^2_{L^2(0,T;H^{k+1}(I_j))}$ in (\ref{eq:tsdifference2}) will facilitate the Gronwall argument in subsection \ref{subsec:gronwall} since it is independent of the current time step; cf. also Lemma \ref{lem:approx}.
\end{remark}

Note that due to summation over all time steps in the anticipated Gronwall argument, using the worst-case estimate (\ref{eq:tsdifference2}) \textit{for each time step} would yield an error of magnitude $h^{k+1}\tau^{-1}$, which is worse than $h^k$ under the considered CFL condition. However, by some additional regularity assumptions on the flux $f$ and the exact solution $u$, we can effectively gauge the number of times for which the estimate (\ref{eq:tsdifference2}) is \textit{actually} necessary. This will be explained in the following.

Let $O_j$ be the set of time levels for which subsequent projections do not coincide in a given element $I_j$, i.e.
\begin{equation}\label{def:O_j}
O_j:=\{n\in\{0,\ldots,N-1\}~:~\Pi^{n+1,\hspace{0.05cm}j}_{h}\neq\Pi^{n,\hspace{0.05cm}j}_{h}\},
\end{equation}
and let $t^{n_0}<t^{n_1}<t^{n_2}$ be a $3$-tuple of time levels such that w.l.o.g.~${f'(u^{n_0}) < -h}$, $f'(u^{n_1}) > h$, $f'(u^{n_2}) < -h$ on $I_j$, meaning that there exist $m,M\in O_j$ with $n_0\leq m\leq n_1\leq M\leq n_2$. Then we have
\begin{align*}
 \frac{4h}{l^2\tau^2} &\leq \left(\frac{f'(u^{n_0}) - f'(u^{n_1})}{t^{n_0}-t^{n_1}} - \frac{f'(u^{n_1}) - f'(u^{n_2})}{t^{n_1} - t^{n_2}}\right)\frac{1}{t^{n_0,n_1}-t^{n_1,n_2}}\\
 &\leq \|\partial_t^2(f'\circ u)\|_{C(Q_T)}=:C_{t,2}\nonumber
\end{align*}
where $t^{n_0,n_1}\in(t^{n_0},t^{n_1}),~t^{n_1,n_2}\in(t^{n_1},t^{n_2})$ are given by the mean value theorem and $l$ denotes the time level distance, i.e.~$l=n_2-n_0$. Since now
\begin{equation*}
\frac{2h^\frac{1}{2}}{\sqrt{C_{t,2}}} \frac{1}{\tau}\leq l,
\end{equation*}
for every $3$-tuple with time level distance $l$ as above, this means that at most $\left\lceil\frac{T}{2}\sqrt{C_{t,2}}h^{-\frac{1}{2}}\right\rceil$ such $3$-tuples can fit into the simulation time window $[0,T]$, yielding
\begin{equation*}
\# O_j\leq 2\cdot\left\lceil\frac{T}{2}\sqrt{C_{t,2}}h^{-\frac{1}{2}}\right\rceil\leq 2T\sqrt{C_{t,2}}h^{-\frac{1}{2}},
\end{equation*}
since the projection changes two times for each such $3$-tuple and $\lceil x\rceil/x\leq 2$ for all $x>0$.

More generally, we obtain the following

\begin{lem}\label{lem:O_j}
	Assume that $f\in C^{\alpha+1}(\R)$ and $u\in C^\alpha([0,T];C_b(\R))$ for some $\alpha\geq 2$. Then
	\begin{equation}
	\#O_j\leq \alpha T \sqrt[\alpha]{C_{t,\alpha}}h^{-\frac{1}{\alpha}}\qquad\text{for all}~ j\in\mathbb{Z},
	\end{equation}
	where $C_{t,\alpha}=\|\partial_t^\alpha(f'\circ u)\|_{C(Q_T)}$.
\end{lem}
But this means that we don't have to give up a whole order of $\tau$ in the Gronwall argument. Instead we will utilize
\begin{lem}\label{lem:approx}
	In the setting of Lemma \ref{lem:O_j} let additionally $u\in H^1([0,T];H^{k+1}(\R))$. Then there holds
	\begin{equation}
	\sum_{n=0}^{N-1} \|\xi_\pi^{n+1} - \xi^n_\pi\|^2\lesssim h^{2k+2-\frac{1}{\alpha}}.
	\end{equation} 
	\begin{proof}
		By definition (\ref{def:O_j}) of $O_j$, Lemma \ref{lem:O_j} and the approximation properties (\ref{eq:tsdifference1}) - (\ref{eq:tsdifference2}), we have
		\begin{align*}
		\sum_{n=0}^{N-1} \|\xi_\pi^{n+1} - \xi^n_\pi\|^2 =& \sum_{n=0}^{N-1} \sum_{j\in\mathbb{Z}} \|\xi_\pi^{n+1} - \xi^n_\pi\|^2_{L^2(I_j)}\\
		\leq& \sum_{j\in\mathbb{Z}} \sum_{n\in O_j}\|\xi_\pi^{n+1} - \xi^n_\pi\|^2_{L^2(I_j)} + \sum_{j\in\mathbb{Z}}\sum_{n\notin O_j}\|\xi_\pi^{n+1} - \xi^n_\pi\|^2_{L^2(I_j)}\\
		\lesssim& \sum_{j\in\mathbb{Z}} \sum_{n\in O_j}h^{2k+2}\left(\|u^0\|^2_{H^{k+1}(I_j)}+T\int_0^T\|\partial_t u\|^2_{H^{k+1}(I_j)}\right)\\ 
		&+\tau h^{2k+2} \|\partial_t u\|^2_{L^2(0,T;H^{k+1}(\R))}\\
		\lesssim&~h^{2k+2}h^{-\frac{1}{\alpha}}.
		\end{align*}
		\end{proof}
\end{lem}

\subsection{Bounding the RHS of the energy identity}
In the following we will assume that the flux and the exact solution are regular enough such that the energy identity (\ref{eq:energyidentity}) holds, e.g.~$f\in C^3(\R)$ and $u\in C^3([0,T];H^1(\R))$. Furthermore we want to use quasi-best approximation properties, i.e.~ we invoke the additional assumption $u\in H^1(0,T;H^{k+1}(\R))$.

A fact that we will use tacitly in the following is that the energy of $\zeta^n_h$ can be approximately bounded by the energy of $\xi^n_h$, more precisely we have
\begin{prop}\label{prop:equiv}
	Under a general CFL condition $\tau\lesssim h$, there holds
	\begin{equation}
	\|\zeta^n_h\|^2\lesssim \|\xi^n_h\|^2+h^{2k+2}.
	\end{equation}
\end{prop}
\begin{proof}
By Lemma 5.3 in \cite{ZhangShu06}, the fractional approximation property (\ref{eq:bestappr}), the inverse property (\ref{eq:fracinv}) together with Youngs inequality, we have
\begin{align}
\|\zeta^n_h-\xi^n_h\|^2 &= \mathcal{K}^n(\zeta^n_h-\xi^n_h) + \tau\mathcal{D}(u^n_h-u^n,\zeta^n_h-\xi^n_h)\nonumber\\
&\leq \epsilon \|\zeta^n_h-\xi^n_h\|^2 + C(\epsilon)\left[ \frac{\tau^2}{h^2}\|\xi^n_h\|^2 + h^{2k}\tau^2 + \frac{\tau^2}{h^{2\lambda}}\|\xi^n_h\|^2 + \frac{\tau^2}{h^\lambda}h^{2k+2-\lambda}\right],
\end{align}	
where the constant $C(\epsilon)>0$ is independent of $n$,$h$ and $\tau$.
Choosing $\epsilon$ small enough, we conclude with the general CFL condition $\tau\lesssim h$.

\end{proof}

We will now proceed to bound the terms on the right-hand side of the energy identity (\ref{eq:energyidentity}). For this we note that there holds
\begin{equation}\label{eq:xi zeta}
	2\|\xi^{n+1}_h - \zeta^n_h\|^2=(\mathcal{L}^n-\mathcal{K}^n)(\xi^{n+1}_h - \zeta^n_h) + \tau\mathcal{D}(e^n_w-e^n_u,\xi^{n+1}_h - \zeta^n_h).
\end{equation}
Accordingly, we will now take measures to control the operators on the right-hand side of (\ref{eq:xi zeta}).

\begin{remark}
	Note that the energy bounds for the time level $t^{n+1}$ will necessitate a minimal accuracy assumption for the previous time level $t^n$, meaning
	\begin{equation}\label{eq:apriori}
	\|u^n-u^n_h\|\lesssim h^{3/2},
	\end{equation}
	which can be proven inductively, i.e.~by the a priori assumption $\|u^0-u^0_h\|\lesssim h^{3/2}$ and using the energy bounds presented below for consecutive time steps.  In particular we thus have that $\|e^n_u\|_\infty+\|e^n_w\|_\infty\lesssim h$. 
\end{remark}

\begin{lem}\label{lem:first bound}
	Let $\epsilon>0$. Then we have, under the general CFL condition $\tau\lesssim h$, that for any $v_h\in V_h$ there holds
	\begin{align}\label{eq:L-K}
		(\mathcal{L}^n-\mathcal{K}^n)(v_h) \leq&~ \epsilon\|v_h\|^2 + C(\epsilon)\{h^{2k+2}\tau + \tau^6 + \tau\frac{\tau}{h}|f'(u^n)|\llbracket\xi^n_h\rrbracket^2+ \tau\frac{\tau}{h}|f'(w^n)|\llbracket\zeta^n_h\rrbracket^2\nonumber\\ &~+ \tau\|\xi^n_h\|^2 + \tau^2\|\partial_x^h(\zeta^n_h-\xi^n_h)\|^2 + \|\xi^{n+1}_\pi-\xi^n_\pi\|^2\},
	\end{align}
	where $|f'(p^n)|\llbracket p^n_h\rrbracket^2:=\sum_{j\in\mathbb{Z}}|f'(p^n_{j})|\llbracket p^n_h\rrbracket^2_{j}$, and $(\partial_x^h p^n_h)\vert_{I_j}:=\partial_x (p^n_h\vert_{I_j})$. Furthermore, there holds
	\begin{align}\label{eq:D1}
	\tau\mathcal{D}(e^n_w-e^n_u,v_h) &\leq \epsilon\|v_h\|^2 +  C(\epsilon)\{h^{2k+2-\lambda}\tau + \frac{\tau^2}{h^\lambda}(|\xi^n_h|_{H^{\lambda/2(\R)}}^2+|\zeta^n_h|_{H^{\lambda/2}(\R)}^2)\},
	\end{align}
	and the constant $C(\epsilon)>0$ is independent of $n$,$h$ and $\tau$.
\end{lem}
\begin{proof}
We obtain (\ref{eq:L-K}) by using Lemma 5.1 from \cite{ZhangShu06} and bounding the flux terms via (\ref{eq:flux}) before applying Young's and an inverse inequality in a straightforward fashion together with the asymptotic equivalence from Proposition \ref{prop:equiv}. Also note that in comparison to \cite{ZhangShu06}, we have to isolate the time step difference $\xi^{n+1}_\pi-\xi^n_\pi$ to treat it by means of Lemma \ref{lem:approx} in the anticipated Gronwall argument.

To show (\ref{eq:D1}), we just use Young's inequality, the fractional approximation property (\ref{eq:bestappr}) together with the inverse property (\ref{eq:fracinv}). 
\end{proof}

\begin{remark}
	The quantity $\frac{\tau^2}{h^\lambda}(|\xi^n_h|_{H^{\lambda/2}(\R)}^2+|\zeta^n_h|_{H^{\lambda/2}(\R)}^2)$ in (\ref{eq:D1}) ought to be absorbed by the dissipative terms on the left-hand side of the energy identity (\ref{eq:energyidentity}). This is achieved by a time step restriction of the form $\tau\leq c h^{\lambda}$ with say $0<c<\frac{c_\lambda}{8}$, which is implied by the general CFL condition $\tau\lesssim h$ for small enough $h$.
\end{remark}

As a consequence, we have
\begin{cor}
Under the general CFL condition $\tau\lesssim h$ there holds,
	\begin{align}\label{eq:xin+1-zetan}
	\|\xi^{n+1}_h-\zeta^n_h\|^2 \lesssim&~ h^{2k+2-\lambda}\tau + \tau^6 +  \tau\frac{\tau}{h}|f'(u^n)|\llbracket\xi^n_h\rrbracket^2+ \tau\frac{\tau}{h}|f'(w^n)|\llbracket\zeta^n_h\rrbracket^2\nonumber\\&+\tau\|\xi^n_h\|^2 + \tau^2\|\partial_x^h(\zeta^n_h-\xi^n_h)\|^2 + \|\xi^{n+1}_\pi-\xi^n_\pi\|^2 \nonumber\\&+ \frac{\tau^2}{h^\lambda}(|\xi^n_h|_{H^{\lambda/2(\R)}}^2+|\zeta^n_h|_{H^{\lambda/2}(\R)}^2).
	\end{align}
\end{cor}

When trying to bound the quantity $\tau^2\|\partial_x^h(\zeta^n_h-\xi^n_h)\|^2$ in (\ref{eq:xin+1-zetan}), a case distinction can be made for $k=1$ and the case $k\geq2$. The following Lemma will facilitate the argument.

\begin{lem}
	Under the general CFL condition $\tau\lesssim h$ there holds,
	\begin{align}\label{eq:Kbound}
	\mathcal{K}^n(v_h) \lesssim& \epsilon \|v_h\|^2 + C(\epsilon)\{h^{2k+2}\tau + \tau\frac{\tau}{h}|f'(u^n)|\llbracket\xi^n_h\rrbracket^2 + \tau\|\xi^n_h\|^2\nonumber\\ &+ \tau\int_{\R}v_h\partial_x^h\xi^n_h~dx\},
	\end{align}
	where the constant $C(\epsilon)>0$ is independent of $n$,$h$ and $\tau$.
\end{lem}
\begin{proof}
	A similar estimate is presented in Lemma 5.3 in \cite{ZhangShu06}. The only difference is that the flux terms are again estimated via (\ref{eq:flux}) before applying Young's and an inverse inequality and the term $(Ch^{-1}+C_\star h^{-1}\|e^n_u\|^2_\infty)h^{2k+2}\tau$ can be improved to $(C+C_\star h^{-1}\|e^n_u\|^2_\infty)h^{2k+2}\tau$ by the strategy that is presented in the proof of Lemma \ref{lem:lastlem} below.
\end{proof}
We proceed to bound the term $\tau^2\|\partial_x^h(\zeta^n_h-\xi^n_h)\|^2$ in (\ref{eq:xin+1-zetan}).

\begin{lem}
		Under the general CFL condition $\tau\lesssim h$ we have for piecewise linear finite elements, i.e.~$k=1$,
		\begin{align}
		\tau^2\|\partial_x^h(\zeta^n_h-\xi^n_h)\|^2 \lesssim&~ h^{4}\tau + \tau\frac{\tau^3}{h^3}|f'(u^n)|\llbracket\xi^n_h\rrbracket^2 + \tau\|\xi^n_h\|^2 \nonumber\\ &+ \frac{\tau^2}{h^\lambda}(h^{4-\lambda}+|\xi^n_h|^2_{H^{\lambda/2}(\R)}),
		\end{align}
		whereas for $k\geq 2$ we have
		\begin{equation}\label{eq:kgeq2}
		\tau^2\|\partial_x^h(\zeta^n_h-\xi^n_h)\|^2\lesssim \frac{\tau^4}{h^4}\|\xi^n_h\|^2 + h^{2k+2-\lambda}\tau.
		\end{equation}
\end{lem}
\begin{proof}
 We start with the case $k=1$. Set $g_h=\zeta^n_h-\xi^n_h$ and let $\overline{g}_h$ be the piecewise constant projection of $g_h$, hence 
	\begin{equation}\label{eq:orth}
		\int_{\R}(g_h-\overline{g}_h)\partial_x^h v_h~dx=0\quad\forall v_h\in V_h^1.
	\end{equation}
By orthogonality (\ref{eq:orth}) and the error equation (\ref{eq:erroreqa}), we have
\begin{equation*}
\|g_h-\overline{g}_h\|^2=(g_h,g_h-\overline{g}_h)=\mathcal{K}^n_h(g_h-\overline{g}_h)+\tau\mathcal{D}(u^n_h-u^n,g_h-\overline{g}_h).
\end{equation*}	
Using a bound for $\mathcal{K}^n$, namely (\ref{eq:Kbound}),
we obtain by orthogonality and for small enough $\epsilon$,
\begin{equation*}
\|g_h-\overline{g}_h\|^2\lesssim  h^{4}\tau + \frac{\tau^2}{h}|f'(u^n)|\llbracket\xi^n_h\rrbracket^2 + \tau\|\xi^n_h\|^2 + \frac{\tau^2}{h^\lambda}(h^{4-\lambda}+|\xi^n_h|^2_{H^{\lambda/2}(\R)}).
\end{equation*}
Thus, by an inverse inequality and the general CFL condition $\tau\lesssim h$, we obtain the bound for $\tau^2\|\partial_x^h(\zeta^n_h-\xi^n_h)\|^2$.

For higher order polynomials, i.e.~$k\geq 2$, the above argument cannot be replicated and one obtains only
\begin{equation}
\tau^2\|\partial_x^h(\zeta^n_h-\xi^n_h)\|^2\lesssim \frac{\tau^4}{h^4}\|\xi^n_h\|^2 + h^{2k+2-\lambda}\tau,
\end{equation}
by the same arguments used in the proof of Proposition \ref{prop:equiv}.
\end{proof}

\begin{remark}
	Note that the term $\frac{\tau^4}{h^4}\|\xi^n_h\|^2$ that arises in (\ref{eq:kgeq2}) necessitates a stricter time step condition $\tau\lesssim h^{4/3}$ for higher order polynomials.
\end{remark}

It thus remains to bound the last three terms on the right-hand side of the energy identity (\ref{eq:energyidentity}). 
\begin{lem}\label{lem:lastlem}
	Let $\epsilon>0$. Under the general CFL condition $\tau\lesssim h$ then there holds,
	\begin{align}
	\mathcal{K}^n(\xi^n_h) + \frac{1}{2}|f'(u^n)|\llbracket\xi^n_h\rrbracket^2\tau \lesssim&~ h^{2k+2}\tau + \tau\|\xi^n_h\|^2,\label{eq:K}\\
	\mathcal{L}^n(\zeta^n_h) + \frac{1}{2}|f'(u^n)|\llbracket\xi^n_h\rrbracket^2\tau  \lesssim&~ h^{2k+2}\tau + \tau^5 + \tau\|\xi^n_h\|^2\nonumber\\ &+ 2\int_{\R}(\xi^{n+1}_\pi - \xi^n_\pi)\zeta^n_h~dx\label{eq:L}, 
	\end{align}
	as well as
	\begin{equation}\label{eq:D}
	\tau|\mathcal{D}(\xi^n_\pi,\xi^n_h)+\mathcal{D}(\zeta^n_\pi,\zeta^n_h)|\leq~ C(\epsilon)h^{2k+2-\lambda}\tau + \epsilon \tau(|\xi^n_h|^2_{H^{\lambda/2}}+|\zeta^n_h|^2_{H^{\lambda/2}}),
	\end{equation}
	where the constant $C(\epsilon)>0$ is independent of $n$,$h$ and $\tau$.
\end{lem}

\begin{proof}
We will prove the inequality (\ref{eq:K}) by following the proof of Lemma 5.4, 5.6 and 5.7 in \cite{ZhangShu06}; the proof of (\ref{eq:L}) works analogously and is thus omitted here. 

To this end, on each node $x_{j+1}$, let us introduce the upwind reference value $p^{n,\star}_h$ of a  function $p^n_h$ ($p^n_h=u^n_h,w^n_h$ or $\Pi_h^nu^n, \Pi_h^nw^n$) at time $t^n$, depending on the sign of $f'(u^n)$ on the adjacent elements $I_{j}\cup I_{j+1}$ in the following way: 
\begin{itemize}
	\item if $f'(u^n)<-h$ on $I_{j}\cup I_{j+1}$, set $p^{n,\star}_h:=p^{n,+}_h$,
	\item else, set $p^{n,\star}_h:=p^{n,-}_h$.
\end{itemize}
	
We then write
\begin{align}
\mathcal{K}^n(\xi^n_h) =& \int_{\R}(\zeta^n_\pi-\xi^n_\pi)\xi^n_h~dx + \tau\int_{\R}(f(u^n_h)-f(u^n))\partial^h_x\xi^n_h~dx\nonumber\\
&+\tau\sum_{j\in\mathbb{Z}}(\hat{h}(u^n_h)-f(u^{n,\star}_h))\llbracket\xi^n_h\rrbracket_j\nonumber + \tau\sum_{j\in\mathbb{Z}}(f(u^{n,\star}_h)-f(u^n))_{j} \llbracket\xi^n_h\rrbracket_j\nonumber\\
\hat{=}&~ W_1 + W_2 + W_3 + W_4.
\end{align}

By quasi-best-approximation we have
\begin{equation}
W_1\lesssim h^{2k+2}\tau + \tau\|\xi^n_h\|^2.
\end{equation}

Next let us study $G^n:=\hat{h}(u^n_h)-f(u^{n,\star}_h)$ on each node $x_{j+1}$ for different instances of the reference value $u^{n,\star}_h$. If $f'(u^n)<-h$ on $I_{j}\cup I_{j+1}$, the reference value is $u^{n,\star}_h=u^{n,+}_h$; considering the reference value  $u^{n,\star}_h=u^{n,-}_h$ is analogous and we omit the details. The value of the upwind flux $\hat{h}(u^n_h)$ however, depends on the sign variation of $f'$ between $u^{n,-}_h$ and $u^{n,+}_h$:
\begin{itemize}
	\item If $f'<0$ between $u^{n,-}_h$ and $u^{n,+}_h$, then $G^n=0$.
	\item If $f'>0$ between $u^{n,-}_h$ and $u^{n,+}_h$, then ${G^n=f(u^{n,-}_h)-f(u^{n,+}_h)}$ and there must be a zero $u^\ast$ between the intervals $[\min(u^{n,-}_h,u^{n,+}_h),\max(u^{n,-}_h,u^{n,+}_h)]$ and $[\inf_{I_j\cup I_{j+1}} u^n,\sup_{I_j\cup I_{j+1}} u^n]$ with $f'(u^\ast)=0$. By the mean value theorem, this implies $|G^n|\leq \|f''\|_\infty \|u^n-u^n_h\|_\infty|\llbracket u^n_h\rrbracket|$ since the difference $|u^n_h-u^\ast|$ is bounded by $|u^n_h-u^n|$.
	\item If $f'$ has a zero between $u^{n,-}_h$ and $u^{n,+}_h$, it follows that $|G^n|\leq\|f''\|_\infty|\llbracket u^n_h\rrbracket|^2$.
\end{itemize}
In total, since $|\llbracket u^n_h\rrbracket|\leq2\|u^n-u^n_h\|_\infty$, it follows that 
\begin{equation*}
|G^n|\lesssim \|u^n-u^n_h\|_\infty |\llbracket u^n_h\rrbracket|.
\end{equation*}
Thus, noticing that $\llbracket u^n_h\rrbracket=\llbracket\xi^n_h\rrbracket-\llbracket\xi^n_\pi\rrbracket$, by Young's and an inverse inequality we get
\begin{align}
|W_3|\lesssim \tau\|u^n-u^n_h\|_\infty(h^{-1}\|\xi^n_h\|^2 + |\xi^n_\pi|_\Gamma^2) \lesssim \tau \|\xi^n_h\|^2 + h^{2k+2}\tau,
\end{align}
where we have used the minimal accuracy assumption (\ref{eq:apriori}) and the quasi-best-approximation property (\ref{eq:bestappr}) in the second inequality.

The contributions $W_2$ and $W_4$ will be studied jointly by using the following Taylor expansions:
In each element,
\begin{align}
f(u^n_h)-f(u^n) =&~ f'(u^n)\xi^n_h + \frac{1}{2}f''(u^n)(\xi^n_h)^2 - f'(u^n)\xi^n_\pi -f''(u^n)\xi^n_h\xi^n_\pi\nonumber\\ &- \frac{1}{2}f'(u^n)(\xi^n_\pi)^2 + \frac{1}{6}f_u'''(\xi^n_h-\xi^n_\pi)^3 \hat{=}~ \phi_1+\ldots+\phi_6,
\end{align}
and, on each node,
\begin{align}
f(u^{n,\star}_h)-f(u^n) =&~ f'(u^n)\xi^{n,\star}_h + \frac{1}{2}f''(u^n)(\xi^{n,\star}_h)^2 - f'(u^n)\xi^{n,\star}_\pi -f''(u^n)\xi^{n,\star}_h\xi^{n,\star}_\pi\nonumber\\ &- \frac{1}{2}f'(u^n)(\xi^{n,\star}_\pi)^2 + \frac{1}{6}f_{u^\star}'''(\xi^{n,\star}_h-\xi^{n,\star}_\pi)^3 \hat{=}~ \psi_1+\ldots+\psi_6,
\end{align}
where $f_u'''$ and $f_{u^\star}'''$ are appropriate intermediate values. We thus have the following representations:
\begin{equation*}
W_2=X_1+\ldots+X_6\quad\text{and}\quad W_4=Y_1+\ldots+Y_6,
\end{equation*}
with
\begin{equation*}
X_i=\tau\int_{\R}\phi_i\partial^h_x\xi^n_h~dx\quad\text{and}\quad Y_i=\tau\sum_{j\in\mathbb{Z}}(\psi_i)_j\llbracket\xi^n_h\rrbracket_j,
\end{equation*}
for $i\in\{1,\ldots,6\}$.

After integrating by parts, we obtain
\begin{equation*}
X_1 + Y_1 = -\frac{\tau}{2}\int_{\R} (\xi^n_h)^2 \partial_x f'(u^n)~dx + \tau \sum_{j\in\mathbb{Z}}f'(u^n)_j(\xi^{n,\star}_h - \overline{\xi}^n_h)_j\llbracket\xi^n_h\rrbracket_j.
\end{equation*} 
By a case distinction regarding the chosen reference values, it is easy to see that either $f'(u^n)_j(\xi^{n,\star}_h - \overline{\xi}^n_h)_j=-\frac{1}{2}|f'(u^n)_j|\llbracket\xi^n_h\rrbracket_j$ or $|f'(u^n)_j-|f'(u^n)_j||\lesssim h$. We thus have
\begin{equation}
X_1+Y_1\leq C\tau\|\xi^n_h\|^2 - \frac{1}{2}|f'(u^n)|\llbracket\xi^n_h\rrbracket^2\tau.
\end{equation}

By inverse estimates and the minimal accuracy assumption (\ref{eq:apriori}), we have
\begin{equation}
X_2+Y_2\lesssim \tau h^{-1}\|\xi^n_h\|_\infty\|\xi^n_h\|^2\lesssim \tau \|\xi^n_h\|^2.
\end{equation}

Next, observe that by definition of the reference value on each node, the term $-f'(u^n)\xi^{n,\star}_\pi$ either vanishes or $|f'(u^n)|\lesssim h$. Thus, by orthogonality (\ref{eq:orth}) and an inverse inequality, we have
\begin{equation}
X_3+Y_3\lesssim \tau h^{2k+2} + \tau\|\xi^n_h\|^2.
\end{equation}

Finally, it is easy to see the remaining estimates,
\begin{align}
X_4+Y_4 &\lesssim \tau\|\xi^n_h\|^2\\
X_5+Y_5 &\lesssim \tau h^{2k+2} + \tau\|\xi^n_h\|^2\\
X_6+Y_6 &\lesssim \tau h^{2k+2} + \tau\|\xi^n_h\|^2,
\end{align} 
by Young's together with the inverse inequalities, quasi-best-approximation (\ref{eq:bestappr}) and the minimal accuracy assumption (\ref{eq:apriori}). This proves (\ref{eq:K}).

The fractional diffusion term is readily handled by Young's inequality together with (\ref{eq:bestappr}) and (\ref{eq:sym}) to yield (\ref{eq:D}). 
\end{proof}

\subsection{The Gronwall argument}\label{subsec:gronwall}

	In order to prepare the last quantity on the right-hand side of (\ref{eq:L}) for the Gronwall argument, we observe
	\begin{align}\label{eq:last term}
	\int_{\R}(\xi^{n+1}_\pi - \xi^n_\pi)\zeta^n_h~dx = & \sum_{j\in\mathbb{Z}}(\mathbf{1}_{O_j}(n)+\mathbf{1}_{O_j^c}(n))\int_{I_j}(\xi^{n+1}_\pi - \xi^n_\pi)\zeta^n_h~dx\nonumber\\ \lesssim &~ \sum_{j\in\mathbb{Z}}\mathbf{1}_{O_j}(n)h^{2k+2-1/\alpha}(\|u^0\|^2_{H^{k+1}(I_j)}+T\|\partial_t u\|^2_{L^2(0,T;H^{k+1}(I_j)})\nonumber\\
	&+\sum_{j\in\mathbb{Z}}\mathbf{1}_{O_j}(n)h^{1/\alpha}\|\xi^n_h\|^2_{L^2(I_j)}+ \tau h^{2k+2} + \tau\|\xi^n_h\|^2
	\end{align}
	where we have used Taylor expansion up to second order, Jensen's inequality, Proposition \ref{prop:equiv} and a splitting with regards to the set of times in $O_j$; cf.~(\ref{def:O_j}).
	Thus, instead of the classical discrete Gronwall framework
	\begin{equation*}
	a^{n+1}+b^n\leq (1+\gamma)a^n + d^n,
	\end{equation*}
	we have to deal with timestep-specific and elementwise effects that can be formalized as
	\begin{equation}\label{eq:Gronwall1}
	\sum_{j\in\mathbb{Z}}a^{n+1}_j + b^n \leq \sum_{j\in\mathbb{Z}} (1+\gamma^n_j)a_j^n + d^n_j + d^n,
	\end{equation}
	where in our setting
	\begin{align*}
	a^n_j&=\|\xi^n_h\|^2_{L^2(I_j)},\quad b^n=c\tau|u^n-u^n_h|^2_{H^{\lambda/2}(R)},\\
	d^n_j&=c\mathbf{1}_{O_j}(n)h^{2k+2-1/\alpha}\left(\|u^0\|^2_{H^{k+1}(I_j)}+T\|\partial_t u\|^2_{L^2(0,T;H^{k+1}(I_j)}\right),\\
	d^n&=c(\tau h^{2k+2} + \tau^5),\quad \gamma^n_j=c(\mathbf{1}_{O_j}(n)h^{1/\alpha}+\mathbf{1}_{O_j^c}(n)\tau)
	\end{align*}
	 for some constant $c>0$ independent of $n,h,\tau$ and $j$.
	Note that (\ref{eq:Gronwall1}) in particular yields
	\begin{equation}\label{eq:Gronwall2}
	\sum_{j\in\mathbb{Z}}\frac{a^{n+1}_j - (1+\gamma^n_j)a^n_j}{\prod_{m=0}^{n}(1+\gamma^m_j)} + b^n \leq \sum_{j\in\mathbb{Z}} d_j^n +d^n.
	\end{equation}
	Now, summing (\ref{eq:Gronwall1}) over $n$ we obtain a telescopic sum, i.e.
	\begin{equation}
	\sum_{j\in\mathbb{Z}}\frac{a^{N}_j}{\prod_{m=0}^{N-1}(1+\gamma_j^m)} + \sum_{n=0}^{N-1}b^n \lesssim a_0 + \sum_{n=0}^{N-1}\sum_{j\in\mathbb{Z}} d_j^n + \sum_{n=0}^{N-1}d^n,
	\end{equation}
	where by invoking Lemma \ref{lem:O_j},
	\begin{equation*}
	\sum_{n=0}^{N-1}\sum_{j\in\mathbb{Z}} d_j^n \lesssim h^{2k+2-2/\alpha},
	\end{equation*}
	and $a_0=0$ by definition of the initial datum for our scheme.
	It remains to use the bound from Lemma \ref{lem:O_j} to multiply both sides by appropriate factors of $(1+c\tau)$ and $(1+c h^{1/\alpha})$. Indeed, according to Lemma \ref{lem:O_j}, we have for all $j\in\mathbb{Z}$,
	\begin{equation*}
	1\leq\frac{(1+c\tau)^{T/\tau}(1+ch^{1/\alpha})^{\alpha T \sqrt[\alpha]{C_{t,\alpha}}h^{-1/\alpha}}}{\prod_{m=0}^{N-1}(1+\gamma_j^m)},
	\end{equation*}
	and thus in total, we get
	\begin{equation}
	a^N + \sum_{n=0}^{N-1}b^n \lesssim e^{c(T+\alpha T \sqrt[\alpha]{C_{t,\alpha}})} (h^{2k+2-2/\alpha} + h^{2k+2-\lambda} + \tau^4),
	\end{equation}
	where $a^N:=\sum_{j\in\mathbb{Z}}a^N_j$,
	which ends the proof of Theorem \ref{thm:main}.8ik
	
	\section*{Acknowledgement}
	
	J.G. is grateful for financial support by the Deutsche Forschungsgemeinschaft project “Dissipative solutions for the Navier-Stokes-Korteweg system and their numerical treatment” 525866748 within SPP 2410 Hyperbolic Balance Laws in Fluid Mechanics: Complexity, Scales, Randomness (CoScaRa). 

\bibliography{references}{}
\bibliographystyle{plain}

\end{document}